\documentclass[english,reqno]{amsart}
\usepackage{etex}
\usepackage{amsmath,amssymb,amsthm,bbm,mathtools,comment}
\usepackage[shortlabels]{enumitem}
\usepackage[pdftex,colorlinks,backref=page,citecolor=blue]{hyperref}
\hypersetup{pdfpagemode=UseNone,pdfstartview={XYZ null null 1.00}}
\usepackage[mathscr]{euscript}
\usepackage[usenames,dvipsnames]{color}
\usepackage{adjustbox,tikz,calc,graphics,babel,standalone}
\usetikzlibrary{shapes.misc,calc,intersections,patterns,decorations.pathreplacing}
\usetikzlibrary{arrows,shapes,positioning}
\usetikzlibrary{decorations.markings}
\usepackage[final]{microtype}
\usepackage[numbers]{natbib}
\usepackage{cmtiup}
\usepackage{amsfonts}
\usepackage{graphicx}
\usepackage{caption}
\usepackage{subcaption}
\usepackage{verbatim}
\usepackage{array}
\usepackage[frame,cmtip,arrow,matrix,line,graph,curve]{xy}
\usepackage{graphpap, color, pstricks}
\usepackage{pifont}
\usepackage[final]{microtype}
\usepackage{cmtiup}

\allowdisplaybreaks

\theoremstyle{plain}
\newtheorem{theorem}{Theorem}[section]

\newtheorem{claim}[theorem]{Claim}

\theoremstyle{remark}
\newtheorem*{remark}{Remark}

\newcommand{\A}{\mathcal{A}}
\newcommand{\B}{\mathcal{B}}

\newcommand{\KN}{[k]^n}

\newcommand{\Z}{\mathbb{Z}}
\newcommand{\E}{\mathbb{E}}
\def\poset{W}
\def\SS{S}
\def\tA{\B}

\DeclareMathOperator{\Aut}{Aut}

\let\emptyset\varnothing

\newcommand{\eps}{\ensuremath{\varepsilon}}

\let\originalleft\left
\let\originalright\right
\renewcommand{\left}{\mathopen{}\mathclose\bgroup\originalleft}
\renewcommand{\right}{\aftergroup\egroup\originalright}

\makeatletter
\def\imod#1{\allowbreak\mkern10mu({\operator@font mod}\,\,#1)}
\makeatother

\title{On symmetric intersecting families of vectors}
\author{Sean Eberhard}
\address{Department of Pure Mathematics and Mathematical Statistics, University of Cambridge, Wilberforce Road, Cambridge CB3\thinspace0WB, UK}
\email{eberhard.math@gmail.com}

\author{Jeff Kahn}
\address{Department of Mathematics, Rutgers University, Piscataway, NJ 08854, USA}
\email{jkahn@math.rutgers.edu}

\author{Bhargav Narayanan}
\address{Department of Mathematics, Rutgers University, Piscataway, NJ 08854, USA}
\email{narayanan@math.rutgers.edu}

\author{Sophie Spirkl}
\address{Department of Mathematics, Rutgers University, Piscataway, NJ 08854, USA}
\email{sophie.spirkl@rutgers.edu}

\date{25 September, 2019}

\begin{document}

\maketitle
\begin{abstract}
A family of vectors in $\KN$ is said to be \emph{intersecting} if any two of its elements agree on at least one coordinate. We prove, for fixed $k \ge 3$, that the size of a symmetric intersecting subfamily of $\KN$ is $o(k^n)$, which is in stark contrast to the case of the Boolean hypercube (where $k =2$). Our main contribution addresses limitations of existing technology: while there are now methods, first appearing in work of Ellis and the third author, for using spectral machinery to tackle problems in extremal set theory involving symmetry, this machinery relies crucially on the interplay between up-sets, biased product measures, and threshold behaviour in the Boolean hypercube, features that are notably absent in the problem considered here. To circumvent these barriers, introducing ideas that seem of independent interest, we develop a variant of the sharp threshold machinery that applies at the level of products of posets.
\end{abstract}

\section{Introduction}
We pursue a line of investigation initiated by Babai~\citep{babai} and Frankl~\citep{frankl-1} about forty years ago concerning the role of symmetry in extremal set theory. Our starting point is the Erd\H{o}s--Ko--Rado theorem~\citep{EKR}, which asserts that for $n,k \in \mathbb{N}$ with $k < n/2$, the largest families of $k$-subsets of $[n]$ are precisely the trivial ones, namely those that consist of all $k$-sets containing some fixed element of $[n]=\{1,2,\dots ,n\}$.  Many variants and generalisations of this theorem (involving different intersection conditions and different discrete structures such as permutations, vectors and graphs) have since been established. A common theme in this line of enquiry is that the extremal constructions are often highly asymmetric, depending  only on a small number of `coordinates'; see~\citep{HM, FranklHM, AK, moon}, for example. It is therefore natural to ask what happens when one further imposes a symmetry requirement on the  family under consideration, the most natural such requirement being that the family be invariant under some transitive subgroup of the symmetric group $\SS_n$.  Indeed, this direction was proposed by Babai \citep{babai} a few decades back, and has since been rather fruitful; see~\citep{frankl-1, frankl-2} for some classical results, and~\citep{ellis, elkal, keith} for more recent developments.

Here, we study intersecting families of vectors. A family $\A \subset \KN$ is said to be \emph{intersecting} if any two of its elements agree on at least one coordinate. Consideration of the orbits of the natural $\Z/k\Z$ action on $\KN$, i.e., the orbits of the map that shifts each coordinate cyclically by one, shows that any intersecting subfamily of $\KN$ has size at most $k^{n-1}$; furthermore, this bound is tight for the trivial family obtained by specifying the value of some fixed coordinate. These observations go back to Berge~\citep{berge} and Livingtson~\citep{living}, and many (more substantial) generalisations are now known; see~\citep{moon, fz, ft, toku, pach} for a sample of the literature.

Here, again, as in our discussion of families of sets, the above extremal examples are highly asymmetric (membership being determined by a single coordinate), though now with a small caveat: in the Boolean hypercube $[2]^n$ with $n$ odd, the family of vectors with more $1$'s than $2$'s is intersecting, of the maximum possible size $2^{n-1}$, and invariant under \emph{all} of $\SS_n$.  However, this has no counterpart for $k\geq 3$, where, even without symmetry, it is known that the only extremal examples are the trivial ones (see~\citep{living}).  In fact, a little thought suggests a more interesting possibility: might it be true that (for $k\geq 3$) symmetric, intersecting families must be \emph{much} smaller?  Our main purpose here is to show that this is indeed the case.

Before stating a precise result, we repeat, a little more formally the definition of symmetry. As above, we use $[n] $ for $ \{1, 2,\dots, n\}$, and $\SS_n$ for the symmetric group on $[n]$, which acts on $\KN$ in the natural way, namely $(\sigma(x))_i =x_{\sigma(i)}$ for $\sigma \in \SS_n$ and $x \in \KN$. The \emph{automorphism group} of $\A \subset \KN$ is, as usual, $\Aut(\A) = \{\sigma \in \SS_n:\sigma(\A) = \A\}$, and we say $\A \subset \KN$ is \emph{symmetric} if $\Aut(\A)$ is a transitive subgroup of $\SS_n$. Our main result is then as follows.
\begin{theorem}\label{mainthm}
There is a universal constant $c>0$ such that the following holds: for each fixed $k \ge 3$, if $\A \subset \KN$ is symmetric and intersecting, then $|\A| = O(k^n / n^{c/k})$.
\end{theorem}
Perhaps surprisingly, even the simple and natural statement that $\A$ as in the theorem must have size $o(k^n)$ seems resistant to elementary proof, and it may be that the more important point of this work is its contribution to methodology. Giving us a starting point, Ellis and the third author~\citep{ellis}, in resolving a conjecture of Frankl~\citep{frankl-1} on symmetric $3$-wise intersecting families, introduced the use of spectral machinery for tackling problems in extremal set theory involving symmetry; this framework has since been successfully adapted --- see~\citep{elkal, keith} --- to resolve other old extremal problems in the Boolean hypercube involving symmetry constraints. Note, though, that this approach depends crucially on the interplay between up-sets, biased product measures, and `sharp threshold' behaviour, all features absent from the problem under consideration here; for example, all of~\citep{ellis, elkal, keith} start with the elementary observation that the $p$-biased measure of an up-set in $[2]^n$ is monotone increasing in $p$, but even this fact that has no useful analogue in $[k]^n$ for $k \ge 3$. This situation is reminiscent of difficulties occasioned by a lack of useful notions of monotonicity in some probabilistic contexts; see the `all blue' problem of~\citep{eyal} for one particularly egregious example.

Here, one could, for example, try working in $[k]^n$ with the natural product order, but one is then confronted with the following obstacles: compressing an intersecting family `upwards' preserves the intersection condition but not the automorphism group, while replacing a family by its `up-closure' preserves symmetries but not the intersection condition; furthermore, there appears to be no natural analogue in $[k]^n$ of the biased product measures on $[2]^n$ that are at the heart of the arguments of~\citep{ellis, elkal, keith}.

Our (at first unpromising-looking) way around these barriers is to embed $[k]^n$ in a larger `covering space', a suitable product of posets, in which up-closure avoids the above difficulties, and on which appropriate analogues of biased product measures may be constructed that still provide the leverage we need. Having transferred our problem to this larger space, we deduce Theorem~\ref{mainthm} using a suitable variant of the sharp threshold theorem of Friedgut and Kalai~\cite{fk}, based, like the original, on the work of Bourgain, Kahn, Kalai, Katznelson, and Linial~\cite{bkkkl}.

The paper is organised as follows. In Section~\ref{sec:fk-for-posets}, we prove our variant of the Friedgut--Kalai sharp threshold theorem for products of posets; the proof of Theorem~\ref{mainthm} follows in Section~\ref{sec:proof}. Finally, we conclude in Section~\ref{sec:conc} with a brief discussion of what might come next.

\begin{remark} After a preliminary version of this manuscript was circulated, it was brought to our attention that Theorem~\ref{mainthm}, with an estimate of $o(k^n)$ in place of $O(k^n / n^{c/k})$, may be deduced from the results of Dinur, Friedgut and Regev~\citep{dfr} on independent sets in graph powers (from which it follows that any intersecting family in $\KN$ is `essentially contained' in an intersecting `junta'). While the main takeaway of this paper might not be new, we believe that the methodology we adopt --- which we see as our main contribution --- is interesting in its own right, having the potential to be brought to bear on other problems in extremal set theory. Secondary benefits of our technique include both effective bounds, as well as short proofs not reliant on the powerful machinery of~\citep{dfr, fr}.
\end{remark}

\section{Biased measures on products of posets}\label{sec:fk-for-posets}

We now present a general construction that is at the heart of our approach.  In what follows, the reader may find it helpful to keep $p$-biased product measures on the Boolean hypercube in mind.

Let $(\poset, \preceq)$ be a finite poset.  We say that $\A \subset \poset$ is an \emph{up-set} if $x \in \A$ and $x \preceq y$ imply $y \in \A$. 
Recall, for probability measures $\mu_0$ and $\mu_1$ on $\poset$, that $\mu_1$ \emph{(stochastically) dominates} $\mu_0$ if $\mu_1(\A) \geq \mu_0(\A)$ for every up-set $\A \subset \poset$. We extend this, saying that  \emph{$\mu_1$ dominates $\mu_0$ with strength $\kappa$} if
\begin{equation} \label{kappa-condition}
	\mu_1(\A) - \mu_0(\A) \geq \kappa \tag{$\dagger$}
\end{equation}
for every up-set $\A \subset \poset$ other than $\emptyset$ and $\poset$. Given probability measures $\mu_0$ and $\mu_1$ on $\poset$, we consider the interpolation from $\mu_0$ to $\mu_1$ --- our analogue of biased product measures --- obtained by taking $\mu_t = (1-t) \mu_0 + t \mu_1$ to be the measure at `time' $t \in [0,1]$. We need the following variant of the Friedgut--Kalai~\citep{fk} theorem; in what follows, as usual, if $\mu$ is a probability measure on $\poset$, then  $\mu^n$ is the corresponding product measure on $\poset^n$.

\begin{theorem}\label{fk-for-posets}
Assume that $\A \subset \poset^n$ is a symmetric up-set, $\mu_0$ and $\mu_1$ are probability measures on 
$\poset$, and $\mu_1$ dominates $\mu_0$ with strength $\kappa>0$. 
If $0 \leq p < q \leq 1$ and $\mu^n_p(\A), \mu^n_q(\A) \in [\eps, 1-\eps]$, then
	\[
		q - p \leq C \kappa^{-1} \log (1 / 2\eps) / \log n,
	\]
	where $C > 0$ is a universal constant.
\end{theorem}
\begin{proof}
We begin with a variant of the Margulis--Russo formula~\citep{margulis,russo}, namely
	\[
		\frac{d}{dp} \mu_p^n(\A) = \sum_{i=1}^n (\mu_p^{i-1} \times (\mu_1 - \mu_0) \times \mu_p^{n-i})(\A).
	\]

	Next, recall that the \emph{influence} $I_{\A,p}(i)$ of a coordinate $i$ is the probability that, for $x \sim \mu^n_p$, changing the value of $x_i$ can affect whether $x \in \A$, i.e., the probability that the `slice'
	\[
		\A_i(x) = \{w \in \poset \colon (x_1, \dots, x_{i-1}, w, x_{i+1}, \dots, x_n) \in \A\}
	\]
	is neither $W$ nor $\emptyset$. By~\eqref{kappa-condition}, 
	\[
		(\mu_p^{i-1} \times (\mu_1 - \mu_0) \times \mu_p^{n-i})(\A) \geq \kappa I_{\A, p}(i),
	\]
	implying 
	\[
		\frac{d}{dp} \mu_p^n(\A) \geq \kappa \sum_{i=1}^n I_{\A, p}(i).
	\]

	On the other hand, as in~\citep{fk}, symmetry and~\citep{bkkkl} give
	\[
		\sum_{i=1}^n I_{\A,p}(i) = \Omega( \min(\mu_p^n(\A), 1-\mu_p^n(\A)) \log n);
	\]
so, combining, we have
	\[
		\frac{d}{dp} \mu_p^n(\A) = \Omega( \kappa \min(\mu_p^n(\A), 1-\mu_p^n(\A)) \log n).
	\]
The stated inequality now follows by elementary calculus.
\end{proof}

\section{Proof of the main result}\label{sec:proof}

As in Theorem~\ref{mainthm}, we assume $\A \subset \KN$ is symmetric and intersecting, and wish to show that $|\A| = o(k^n)$. In outline, the proof of this fact goes as follows.

\begin{enumerate}
	\item Enlarge $\KN$ to a space $\poset^n$, where $\poset$ is a suitably chosen `covering poset', equipped with an appropriate $\mu_0$ and $\mu_1$.
	\item Use the fact that $\A$ is intersecting to conclude that its up-closure in $W^n$ has $\mu_t$-measure at most 1/2 for a suitable time $t$ (in the interpolation from $\mu_0$ to $\mu_1$).
	\item Deduce from Theorem~\ref{fk-for-posets} that $\A$ must have been vanishingly small in the original space $\KN$.
\end{enumerate}

\begin{proof}[Proof of Theorem~\ref{mainthm}]
Write $[k]^{(r)}$ for the collection of $r$-subsets of $[k]$, and let
$(\poset, \preceq)$ be the poset
	\[
		\poset = [k]^{(1)} \cup [k]^{(k-1)},
	\]
with $\preceq$ defined by inclusion. We embed $[k]$ in $W$ by identifying $[k]$ with $[k]^{(1)}$ in the obvious way.

Let $\mu_0$ and $\mu_1$ be, respectively, the uniform (probability) measures on $[k]^{(1)}$ and
$[k]^{(k-1)}$, and,
as in Section~\ref{sec:fk-for-posets}, set $\mu_t = (1-t)\mu_0 + t\mu_1$, noting that $\mu_{1/2}$ is the uniform measure on $\poset$.

	\begin{claim}
$\mu_1$ dominates $\mu_0$ with strength $1/k$.
	\end{claim}
	\begin{proof}
		Let $\A \subset \poset$ be a proper, nontrivial up-set. If $\A \subset [k]^{(k-1)}$ or $\A \supset [k]^{(k-1)}$, then it is clear that
		\[
			(\mu_1 - \mu_0)(\A) \geq 1/k.
		\]
		The only other possibilities are the `stars'
		\[
			\A = \{\{i\}\} \cup ( [k]^{(k-1)} \setminus \{[k]\setminus\{i\}\}),
		\]
		for which we have
		\[
			(\mu_1 - \mu_0)(\A) = 1 - 2/k \geq 1/k. \qedhere
		\]
	\end{proof}

	We now extend the notion of an intersecting family to $\poset^n$ by saying that $\A \subset \poset^n$ is \emph{intersecting} if for any $x,y \in \A$, there is some $i \in [n]$ such that $x_i \cap y_i \neq \emptyset$.

	\begin{claim}\label{C3.2}
		If $\A \subset \poset^n$ is intersecting, then $\mu_{1/2}(\A) \leq 1/2$.
	\end{claim}
	\begin{proof}
		Note that if $x \sim \mu_{1/2}$, then we also have $x^c \sim \mu_{1/2}$, where $x^c = (\bar x_1, \bar x_2, \dots, \bar x_n)$ is the point-wise complement of $x$. Since at most one of $x$ and $x^c$ can belong to $\A$, we have
		\[
			2\mu_{1/2}(\A) = \E [|\A \cap \{x, x^c\}|] \leq 1.\qedhere
		\]
	\end{proof}

We may now finish as follows. Given $\A \subset \KN$ symmetric and intersecting as in Theorem~\ref{mainthm}, let $\tA$ be its up-closure in $\poset^n$ and notice that $\tA$ is again symmetric and intersecting.  Claim~\ref{C3.2} thus gives $\mu_{1/2}(\tA) \leq 1/2$, so, applying Theorem~\ref{fk-for-posets} with $p=0$, $q=1/2$, $\eps = \mu_0(\tA)$ and $\kappa = 1/k$,
we have
	\[
		1/2 \leq C \kappa^{-1} \log(1/2\eps) / \log n,
	\]
or, rearranging, 
	\[
\frac{|\A|}{k^n} = \mu_0(\tA) \leq \frac{n^{-\kappa/2C}}{2} =o(1).\qedhere
	\]
\end{proof}
\section{Conclusion}\label{sec:conc}

The most obvious question raised by the present work is that of estimating more accurately --- beyond the $o(k^n)$ of Theorem~\ref{mainthm} --- how large a symmetric, intersecting subfamily of $[k]^n$ can really be (for $k\geq 3$).  The best examples $\A$ that we know are \emph{set-intersecting}, in the sense that there is a symmetric, intersecting family $\B$ of \emph{subsets} of $[n]$ and an $\ell\in [k]$ such that $x\in [k]^n$ belongs to $\A$ if and only if there is some $B\in\B$ such that $x_i=\ell$ for all $i\in B$. For instance, if $n=q^2+q+1$ with $q$ a prime power, then we may take $\B$ to be the set of lines of the classical projective plane $PG(2,q)$ (see~\citep{vw}, for instance), yielding an $\A$ of size roughly $k^{n-\sqrt{n}}$. Note that the family consisting of all strings with 1's in more than half the coordinates, the counterpart of the exceptional example for $[2]^n$ mentioned in the introduction, does much worse.

It seems possible (but maybe impossible to prove) that the largest symmetric intersecting families in $[k]^n$ are set-intersecting. Failing that, it would be very interesting to at least show that there are constants $c,\delta>0$ (possibly depending on $k$) such that for any symmetric intersecting $\A \subset [k]^n$, we have
\[\log_k |\A| \le n - cn^{\delta}.\]

Finally, we expect that the main technical contribution of this paper --- dealing with intersection problems by situating them in a suitable covering space --- will be applicable to further questions in extremal set theory; we hope to return to this circle of ideas in future work.  

\section*{Acknowledgements}
The second author was supported by NSF grant DMS-1501962, the third by NSF grant DMS-1800521, and the fourth by NSF grant DMS-1802201.

\bibliographystyle{amsplain}
\bibliography{sym_vec_fam}

\providecommand{\bysame}{\leavevmode\hbox to3em{\hrulefill}\thinspace}
\providecommand{\MR}{\relax\ifhmode\unskip\space\fi MR }
\providecommand{\MRhref}[2]{%
  \href{http://www.ams.org/mathscinet-getitem?mr=#1}{#2}
}
\providecommand{\href}[2]{#2}
\begin{thebibliography}{10}

\bibitem{AK}
R.~Ahlswede and L.~H. Khachatrian, \emph{The complete intersection theorem for
  systems of finite sets}, European J. Combin. \textbf{18} (1997), 125--136.

\bibitem{babai}
L.~Babai, Personal communication.

\bibitem{berge}
C.~Berge, \emph{Nombres de coloration de l'hypergraphe {$h$}-parti complet},
  Hypergraph {S}eminar ({P}roc. {F}irst {W}orking {S}em., {O}hio {S}tate
  {U}niv., {C}olumbus, {O}hio, 1972; dedicated to {A}rnold {R}oss), 1974,
  pp.~13--20. Lecture Notes in Math., Vol. 411.

\bibitem{bkkkl}
J.~Bourgain, J.~Kahn, G.~Kalai, Y.~Katznelson, and N.~Linial, \emph{The
  influence of variables in product spaces}, Israel J. Math. \textbf{77}
  (1992), 55--64.

\bibitem{frankl-2}
P.~J. Cameron, P.~Frankl, and W.~M. Kantor, \emph{Intersecting families of
  finite sets and fixed-point-free {$2$}-elements}, European J. Combin.
  \textbf{10} (1989), 149--160.

\bibitem{dfr}
I.~Dinur, E.~Friedgut, and O.~Regev, \emph{Independent sets in graph powers are
  almost contained in juntas}, Geom. Funct. Anal. \textbf{18} (2008), 77--97.

\bibitem{elkal}
D.~Ellis, G.~Kalai, and B.~Narayanan, \emph{On symmetric intersecting
  families}, Preprint, arXiv:1702.02607.

\bibitem{ellis}
D.~Ellis and B.~Narayanan, \emph{On symmetric 3-wise intersecting families},
  Proc. Amer. Math. Soc. \textbf{145} (2017), 2843--2847.

\bibitem{EKR}
P.~Erd\H{o}s, C.~Ko, and R.~Rado, \emph{Intersection theorems for systems of
  finite sets}, Quart. J. Math. Oxford \textbf{12} (1961), 313--320.

\bibitem{frankl-1}
P.~Frankl, \emph{Regularity conditions and intersecting hypergraphs}, Proc.
  Amer. Math. Soc. \textbf{82} (1981), 309--311.

\bibitem{FranklHM}
\bysame, \emph{Erd{\H{o}}s-{K}o-{R}ado theorem with conditions on the maximal
  degree}, J. Combin. Theory Ser. A \textbf{46} (1987), 252--263.

\bibitem{fz}
P.~Frankl and Z.~F\"{u}redi, \emph{The {E}rd{\H o}s-{K}o-{R}ado theorem for
  integer sequences}, SIAM J. Algebraic Discrete Methods \textbf{1} (1980),
  376--381.

\bibitem{ft}
P.~Frankl and N.~Tokushige, \emph{The {E}rd{\H o}s-{K}o-{R}ado theorem for
  integer sequences}, Combinatorica \textbf{19} (1999), 55--63.

\bibitem{keith}
K.~Frankston, J.~Kahn, and B.~Narayanan, \emph{On regular 3-wise intersecting
  families}, Proc. Amer. Math. Soc. \textbf{146} (2018), 4091--4097.

\bibitem{fk}
E.~Friedgut and G.~Kalai, \emph{Every monotone graph property has a sharp
  threshold}, Proc. Amer. Math. Soc. \textbf{124} (1996), 2993--3002.

\bibitem{fr}
E.~Friedgut and O.~Regev, \emph{Kneser graphs are like {S}wiss cheese},
  Discrete Anal. \textbf{18} (2018), Paper No. 2.

\bibitem{HM}
A.~J.~W. Hilton and E.~C. Milner, \emph{Some intersection theorems for systems
  of finite sets}, Quart. J. Math. Oxford, Series 2 \textbf{18} (1967),
  369--384.

\bibitem{living}
M.~L. Livingston, \emph{An ordered version of the {E}rd{\H o}s-{K}o-{R}ad\'{o}
  theorem}, J. Combin. Theory Ser. A \textbf{26} (1979), 162--165.

\bibitem{eyal}
E.~Lubetzky, \emph{Dynamics for the critical 2d potts/fk model: many questions
  and a few answers}, Charles River Lecture Series (2018).

\bibitem{margulis}
G.~A. Margulis, \emph{Probabilistic characteristics of graphs with large
  connectivity}, Problemy Pereda\v ci Informacii \textbf{10} (1974), 101--108.

\bibitem{moon}
A.~Moon, \emph{An analogue of the {E}rd{\H o}s-{K}o-{R}ado theorem for the
  {H}amming schemes {$H(n,\,q)$}}, J. Combin. Theory Ser. A \textbf{32} (1982),
  386--390.

\bibitem{pach}
J.~Pach and G.~Tardos, \emph{Cross-intersecting families of vectors}, Graphs
  Combin. \textbf{31} (2015), 477--495.

\bibitem{russo}
L.~Russo, \emph{An approximate zero-one law}, Z. Wahrsch. Verw. Gebiete
  \textbf{61} (1982), 129--139.

\bibitem{toku}
N.~Tokushige, \emph{Cross {$t$}-intersecting integer sequences from weighted
  {E}rd{\H o}s-{K}o-{R}ado}, Combin. Probab. Comput. \textbf{22} (2013),
  622--637.

\bibitem{vw}
J.~H. van Lint and R.~M. Wilson, \emph{A course in combinatorics}, second ed.,
  Cambridge University Press, Cambridge, 2001.

\end{thebibliography}
\end{document}